\documentclass[12pt,reqno]{article}
\usepackage[usenames]{color}
\usepackage{amssymb}
\usepackage{graphicx}
\usepackage{amscd}
\usepackage{mathdots}
\usepackage[colorlinks=true,
linkcolor=webgreen,
filecolor=webbrown,
citecolor=webgreen]{hyperref}

\definecolor{webgreen}{rgb}{0,.5,0}
\definecolor{webbrown}{rgb}{.6,0,0}

\usepackage{color}
\usepackage{fullpage}
\usepackage{float}

\usepackage{graphics,amsmath,amssymb}
\usepackage{amsthm}
\usepackage{amsfonts}
\usepackage{latexsym}
\usepackage{epsf}

\newcommand{\seqnum}[1]{\href{http://oeis.org/#1}{\underline{#1}}}

\makeatletter
\newcommand*\bigcdot{\mathpalette\bigcdot@{.5}}
\newcommand*\bigcdot@[2]{\mathbin{\vcenter{\hbox{\scalebox{#2}{$\m@th#1\bullet$}}}}}
\makeatother

\begin{document}
\theoremstyle{plain}
\newtheorem{theorem}{Theorem}
\newtheorem{corollary}[theorem]{Corollary}
\newtheorem{lemma}[theorem]{Lemma}
\newtheorem{proposition}[theorem]{Proposition}

\theoremstyle{definition}
\newtheorem{definition}[theorem]{Definition}
\newtheorem{example}[theorem]{Example}
\newtheorem{conjecture}[theorem]{Conjecture}

\theoremstyle{remark}
\newtheorem{remark}[theorem]{Remark}
\newtheorem*{remark-non}{Remark}

\begin{center}
\vskip 1cm{\Large\bf Fibonacci and Lucas Riordan Arrays and \\
\vskip .05in
Construction of Pseudo-Involutions }
\\
Candice Marshall and Asamoah Nkwanta
\\
Morgan State University, Baltimore, MD 21251
\end{center}

\vskip .2 in

\begin{abstract}
Riordan arrays, denoted by pairs of generating functions $(g(z), f(z))$, are infinite lower-triangular matrices that are used as combinatorial tools. 
In this paper, we present Riordan and stochastic Riordan arrays that have connections to the Fibonacci and modified Lucas numbers.  Then, we present some pseudo-involutions in the Riordan group that are based on constructions starting with a certain generating function $g(z)$. We also present a theorem that shows how to construct pseudo-involutions in the Riordan group starting with a certain generating function $f(z)$ whose additive inverse has compositional order $2$. The theorem is then used to construct more pseudo-involutions in the Riordan group where some arrays have connections to the Fibonacci  and modified Lucas numbers. A MATLAB algorithm for constructing the pseudo-involutions is also given.
\end{abstract}

\section{Introduction}

In this paper, we present examples of Riordan and stochastic Riordan arrays that have connections to the Fibonacci \seqnum{A000045} \cite{sloane} and modified Lucas numbers \seqnum{A000204} \cite{sloane}.  
In particular, we present some pseudo-involutions in the Riordan group; the first array that is a pseudo-involution includes the Fibonacci numbers in the 
first column,  the second has the modified Lucas numbers in the first column, and the third contains the convolved Fibonacci numbers \seqnum{A001629} \cite{sloane} in the first column. See OEIS \cite{sloane} for more information on the convolved Fibonacci numbers. The constructions start with a certain generating function $g(z)$. We also present a theorem for constructing 
Riordan group pseudo-involutions starting with a certain generating function $f(z)$ whose additive inverse has compositional order $2$.  The theorem is then used to construct more pseudo-involutions in the Riordan group. Other constructions of pseudo-involutions in the Riordan group can be found in \cite{barry2, luzon, marshall1}.

\section{Preliminaries}

We begin with presenting some preliminary information. For those readers familiar with Riordan arrays, you may skip this section and start with Section 3. 

The Fibonacci numbers are given by the following recurrence relation, $\mathcal{F}_0 = \mathcal{F}_1 = 1$ and $\mathcal{F}_n = \mathcal{F}_{n-1} + \mathcal{F}_{n-2}$ \cite{rob}. The first few Fibonacci numbers \seqnum{A000045} \cite{sloane} are 
1, 1, 2, 3, 5, 8, 13, ... The $n^{th}$ Fibonacci number is given by the following expression
\[\mathcal{F}_n = \tfrac{ \left (\tfrac{1+\sqrt{5}}{2} \right )^{n+1} - \left ( \tfrac{1-\sqrt{5}}{2} \right )^{n+1}}{\sqrt{5}}. \] The generating function of the Fibonacci numbers is 
\[ \mathcal{F}(z) = \tfrac{1}{1-z-z^2} = \sum_{n \geq 0} \mathcal{F}_nz^n.\] In this paper, we also use the Lucas numbers, sequence \seqnum{A000032} \cite{sloane}. However, we use a modified version, denoted by  $\mathcal{L}_n$, because the Lucas numbers begin with a $2$. We prefer that the leading (constant) term be $1$. 
The corresponding generating function of the modified Lucas numbers \seqnum{A000204} \cite{sloane} is 
\[ \mathcal{L}(z) = \tfrac{1+z^2}{1-z-z^2} = \sum_{n \geq 0} \mathcal{L}_nz^n.\]

We now introduce the definition of Riordan arrays. The arrays are known as Riordan matrices.

\begin{definition} \cite{shap}
An infinite matrix $L=(l_{n,k})_{n,k\in \mathbb{N}^{\ast }}$ with entries in 
$\mathbb{C}$ is called a Riordan matrix if the $k^{th}$ column satisfies%
\begin{equation*}
\sum_{n\geq 0}l_{n,k}\ z^{n}=g(z)f(z)^{k}
\end{equation*}%
where $g(z)=g_{0}+g_{1}z+g_{2}z^{2}+\cdots $ and $f(z)=f_{1}z+f_{2}z^{2}+%
\cdots $ belong to the ring of formal power series $\mathbb{C}[[z]],$ and $%
f_{1}\neq 0$ and $g_{0}\neq 0.$
\end{definition}

A Riordan matrix, denoted by $L$, is usually written in pair form as $\left(
g(z),f(z)\right) $ or $(g,f)$.

\begin{example}A typical example of a Riordan matrix is the Pascal matrix. In this case, $g(z) = \frac{1}{1-z} $ and $f(z) = \frac{z}{1-z}$. In pair form notation, $P = \left (\frac{1}{1-z},\frac{z}{1-z} \right )$ where the entries of the Pascal Riordan matrix are Pascal's triangle \seqnum{A007318} \cite{sloane} written in infinite lower-triangular form. \end{example}

The following theorem is called the Fundamental Theorem of the Riordan
Group. It leads to the next theorem by applying the fundamental theorem to
an arbitrary Riordan matrix $N$, one column of $N$ at a time.

\begin{theorem}
\cite{shap} If $L=(l_{n,k})_{n,k\in \mathbb{N}^{\ast }}=\left( g(z),f(z)\right) $
is a Riordan matrix and $h(z)$ is the generating function of the sequence
associated with the entries of the column vector \\ $h=(h_{k})_{k\in \mathbb{N}%
^{\ast }},$ then the product of $L$ and $h(z)$, defined by%
\begin{equation*}
L\otimes h(z)=g(z)h(f(z)),
\end{equation*}%
is the generating function of the sequence associated with the entries of
the column vector $\left( \sum_{k=0}^{n}l_{n,k}h_{k}\right) _{n\in \mathbb{N}%
^{\ast }}.$
\end{theorem}

Let us denote by $L * N$, or by simple juxtaposition $LN$, the row-by column
product of two Riordan matrices.

\begin{theorem}
\cite{shap} If 
\begin{equation*}
L=(l_{n,k})_{n,k\in \mathbb{N}^{\ast }}=\left( g(z),f(z)\right)
\end{equation*}%
and 
\begin{equation*}
N=(\nu _{n,k})_{n,k\in \mathbb{N}^{\ast }}=\left( h(z),l(z)\right)
\end{equation*}%
are Riordan matrices, then

\begin{align*}
L\ast N=\left( \sum_{j=0}^{n}l_{n,j}\nu _{j,k}\right) _{n,k\in \mathbb{N}
^{\ast }}& =\left( g(z),f(z)\right) \ast \left( h(z),l(z)\right) \\
& =\left( g(z)h(f(z)),l(f(z))\right) ,
\end{align*}%
and the set $(\mathcal{R},\ast )$ of all Riordan matrices is a group under
the operation of matrix multiplication.
\end{theorem}

A Riordan matrix is also known as a proper Riordan array. A Riordan array which is not proper does not belong to the Riordan group. Given that $L = (g(z), f(z))$ is a Riordan matrix, then the identity element is $(1,z)$ and the inverse of $(g(z), f(z))$ is $\left ( \tfrac{1}{g(\bar{f}(z))}, \bar{f}(z) \right )$, where $\bar{f}(z)$ is the compositional inverse of $f(z)$. 
Some well known subgroups of the Riordan group mentioned in this paper are given in \cite{shap}.
 A special property of Riordan matrices is that every element (except those in the first row and first column) can be expressed as a linear combination of the elements in the preceding row starting from the preceding column \cite{rogers}. Also, every element in the first column other than the first element, can be expressed as a linear combination of all the elements of the preceding row \cite{Merlini}. 
 These properties are stated in the following theorem.
\begin{theorem} \cite{Merlini, rogers} Let $L = (l_{n,k})$ be an infinite triangular matrix. Then $L$ is a Riordan matrix if and only if there exists two sequences $A = {a_0, a_1, a_2, ...}$ and $Z = {z_0, z_1, z_2, ...}$ with $a_0 \neq 0, z_0 \neq 0$ such that 
 \[  l_{n+1,k+1} = \sum_{j=0}^\infty a_jl_{n, k+j} \hspace{0.1in} (k,n = 0,1...), \] and
\[   l_{n+1, 0} = \sum_{j=0}^\infty z_jl_{n,j} \hspace{0.1in} (n = 0, 1,...). \]
\end{theorem}

These sequences are called the A and Z sequences. We can find the A and Z sequences of a Riordan matrix by calculating its production matrix. The production matrix is the product of the inverse of the Riordan matrix and the Riordan matrix with the first row removed \cite{barry}.

We now introduce a few definitions and propositions about pseudo-involutions that will be useful in the next section. 

\begin{definition}
A stochastic array is an array whose row sums equal one.
\end{definition}

\begin{definition} \cite{jean2}
A stochastic Riordan array is an array which can be written as a pair of generating functions $(g(z), f(z))$ with row sums equal to one. A stochastic Riordan matrix is a proper stochastic Riordan array. It is invertible and therefore belongs to the Riordan group.
\end{definition}

\begin{definition}
\cite{nkwanta} An element $L$ of the Riordan group is called a pseudo involution or
is said to have pseudo-order $2$ if $LM$ or equivalently $ML$ has order $2$
where $M = (1, -z)$. The entries of $M$ consist of alternating $1$ and $-1$ on the main diagonal and $0$'s everywhere else. 
\end{definition}

One interesting property of pseudo-involutions is that one can calculate its inverse very easily by placing negative signs on alternate column and row entries. This is given by the following proposition.
\begin{proposition} \cite{cheon, jean}
If $A$ is a Riordan matrix that is a pseudo-involution, then \\ $A^{-1} = MAM.$ 
\end{proposition}

\begin{example} Consider the Pascal Riordan matrix given in Example 2. Since P is a pseudo-involution \cite{nkwanta}, then the first few entries of $P^{-1}$ \seqnum{A130595} \cite{sloane} are shown below:

\begin{center}  $\begin{bmatrix} \phantom{-}1\\-1&\phantom{-}1\\\phantom{-}1&-2&\phantom{-}1\\-1&\phantom{-}3&-3&\phantom{-}1\\\phantom{-}1&-4&\phantom{-}6&-4&\phantom{-}1\\-1&\phantom{-}5&-10 &\phantom{-}10&-5&\phantom{-}1\\\phantom{-}1&-6&\phantom{-}15&-20&\phantom{-}15&-6&\phantom{-}1& \end{bmatrix}$. \end{center}

\end{example}

\section{Pseudo-involutions and Subgroups of the Riordan Group}

\begin{proposition} A Riordan matrix $(g(z), F(z))$ is an involution if and only if \[g(z) \cdot g(F(z)) = 1\] and \[F(F(z)) = z. \] \end{proposition}

\begin{proof} 
\begin{align*} (g(z), F(z)) \hspace{0.02in} \text{is an involution}  &\iff (g(z), F(z)) * (g(z), F(z)) = (1,z) \\
& \iff  ( g(z) \cdot g(F(z)), F(F(z)) ) = (1,z). \end{align*}
\end{proof}
 
 Letting $F(z) = -f(z)$, we get the following corollary.
 \begin{corollary} A Riordan matrix $(g(z), -f(z))$ is a pseudo-involution if and only if \[g(z) \cdot g(-f(z)) = 1\] and \[-f(-f(z)) = z. \] \end{corollary}
 
 \begin{proposition} A Riordan matrix denoted by $A$ is a pseudo-involution if and only if $A^{-1}$ is a pseudo-involution.\end{proposition} 

We conclude this section by presenting some simple facts about certain subgroups of the Riordan group and pseudo-involutions. 

\begin{proposition} An element of the Appell subgroup, $(g(z), z)$, is a pseudo-involution if and only if $g(z)\cdot g(-z) = 1$. \end{proposition}
\begin{proof}
\begin{align*}
\big ( (g(z), z) * (1, -z) \big )^2 = (1, z) & \iff (g(z), -z)^2 = (1, z) \\
& \iff (g(z) \cdot g(-z), z) = (1, z) \\
& \iff g(z) \cdot g(-z) = 1. 
\end{align*}
\end{proof}

\begin{example} Let $g(z) = \tfrac{1+kz}{1-kz}$, where $k$ is any real number. Then,  $g(-z) = \tfrac{1-kz}{1+kz}$ and so $\tfrac{1+kz}{1-kz} \cdot \tfrac{1-kz}{1+kz} = 1.$ Hence, $\left(\tfrac{1+kz}{1-kz}, z \right )$ is a pseudo-involution in the Appell subgroup for any real number $k$.
\end{example}

\begin{proposition}
An element of the Bell subgroup, $\left (\tfrac{f(z)}{z}, f(z) \right )$, is a pseudo-involution iff $-f(z)$ has compositional order $2$. 
\end{proposition}

\begin{proposition}
An element of the Associated subgroup, $(1, f(z))$, is a pseudo-involution iff $-f(z)$ has compositional order $2$.
\end{proposition}

\begin{proposition}
An element of the Derivative subgroup, $(f'(z), f(z))$, is a pseudo-involution iff $-f(z)$ has compositional order $2$. 
\end{proposition}

\begin{proposition}
An element of the Hitting time subgroup, $(\tfrac{f'(z)}{f(z)}, f(z))$,  is a pseudo-involution iff $-f(z)$ has compositional order $2$. 
\end{proposition}

\section{Lucas Riordan Array}

\subsection{Stochastic Lucas Array}

In this section, we create a stochastic Lucas array based on the modified Lucas numbers and then generate a stochastic Riordan matrix. 
\begin{lemma} \cite{jean} If a Riordan array is stochastic, then $f(z) = -g(z) +zg(z) +1.$ \end{lemma} 
Using this lemma, we can generate a stochastic array for any sequence of counting numbers.

\begin{example} Using the generating function of the modified Lucas numbers as $g(z)$, we can generate the stochastic Lucas array. Since $g(z) = \tfrac{1+z^2}{1-z-z^2}$, by Lemma 20 \[f(z) = \tfrac{-2z^2 + z^3}{1-z-z^2}.\] Thus, the stochastic Lucas array is \[(g(z), f(z)) = \left (\frac{1+z^2}{1-z-z^2},\frac{-2z^2+z^3}{1-z-z^2} \right ) \] where the first few entries of the array are shown below:
\[\begin{bmatrix}1\\1&\phantom{-}0\\3&-2&0\\4&-3&0&\phantom{-}0\\7&-10&4&\phantom{-}0&0\\11&-18&8&\phantom{-}0&0&0&\\18&-38&29&-8&0&0&0\\29&-71&63&-20&0&0&0&0\\47&-134&150&-78&16&0&0&0&0\\76&-245&317&-195&48&0&0&0&0&0 \end{bmatrix}.\]
\end{example}
 
This is an example of a vertically stretched Riordan array. After removing the leading 1 from the generating function of the modified Lucas numbers by subtracting 1 and dividing by $z$, and applying Lemma 20, we can generate a stochastic Lucas matrix $\left (\frac{1+2z}{1-z-z^2},\frac{-2z+z^2}{1-z-z^2} \right )$. The first few entries are shown below:
 
 \[ \begin{bmatrix}1\\3&-2\\4&-7&4\\7&-14&16&-8\\11&-31&41&-36&16\\18&-60&105&-110&80&-32&\\29&-116&235&-315&280&-176&64\\47&-216&512&-790&880&-688&384&-128\\76&-397&1063&-1894&2425&-2344&1648&-832&256\\123&-718&2153&-4298&6303&-7002&6032&-3872&1792&-512 \end{bmatrix}.\]

The Z sequence of the stochastic Lucas matrix is infinite. The first few terms are:
\[3, \tfrac{5}{2}, \tfrac{25}{8}, \tfrac{25}{8}, \tfrac{375}{128}, \tfrac{375}{128}, \tfrac{3125}{1024}, \tfrac{3125}{1024}, ... \]

The A sequence is also infinite. The first few terms are:
\[-2, \tfrac{1}{2}, -\tfrac{5}{8}, 0, \tfrac{25}{128}, 0, -\tfrac{125}{1024}, 0, ... \]

It is easy to confirm the above new array and matrix are stochastic. \\ See \cite{jean} for an example of a stochastic Fibonacci array.

\section{Construction of a Lucas Pseudo-involution}

The following theorem is used to construct a Lucas pseudo-involution. 

\begin{theorem}\cite{marshall} If $g(z) = g_0+g_1z+g_2z^2+..., $ with $g_0 = 1$ and $g_1 \neq 0$, then there exists a unique $f(z)$ such that $(g(z), f(z))$ is a pseudo-involution. In fact, 

\[f(z) = -\bar{G} \big ( \tfrac {-G(z)}{g(z)} \big ) \]

\noindent where $G(z) = g(z) - g_0$. \end{theorem}

\subsection{Lucas Pseudo-involution}

The next array is another pseudo-involution in the Riordan group. The modified Lucas numbers make up the first column of the array. 
  
\begin{example}
We use the generating function of the modified Lucas numbers and Theorem 22 to generate another pseudo-involution. Let $g(z) = \tfrac{1+z^2}{1-z-z^2}$. Then, 
\[G(z) = \tfrac{z+2z^2}{1-z-z^2}\] 
and 
\[ \bar{G}(z) = \tfrac{-(1+z) + \sqrt{5z^2+10z+1}}{2(2+z)} \]
Now, 
\[ \tfrac{1}{g(z)} -1 = \tfrac{-z-2z^2}{1+z^2} \]
Hence, 
\[ F(z) = \bar{G}(\tfrac{-z-2z^2}{1+z^2}) = \tfrac{-1+z+z^2+\sqrt{z^4+10z^3-13z^2-10z+1}}{4-2z}.  \]

So, \[  (g(z), f(z)) = \Big ( \tfrac{1+z^2}{1-z-z^2}, \tfrac{1-z-z^2-\sqrt{z^4+10z^3-13z^2-10z+1}}{4-2z} \Big ) \] is a pseudo-involution of the Riordan group with the modified Lucas numbers in the first column. The first few entries of this Riordan matrix are shown below:

$$
\begin{bmatrix}
1 \\
1&1 \\
3&6&1 \\
4&33&11&1\\
7&214&88&16&1\\
11&1572&699&168&21&1\\
18&12686&5787&1584&273&26&1\\
29&108583&50036&14652&2994&403&31&1\\
47&967294&447998&136436&30792&5054&558&36&1
\end{bmatrix}.
$$
\vspace{0.1in}

\end{example}
 
 This can be confirmed by direct calculations.
 The Lucas pseudo-involution also has infinite A and Z sequences. Its Z sequence is \[1,2,-11,58,-384, 2872, -23416, 201608, ... \] and its A sequence is \[1,5,0,45,-225, 1980, -16200, 142920, ...\]
 
 \vspace{0.1in}
 
\subsection{Convolved Fibonacci Pseudo-involution}

The generating function of the convolved Fibonacci numbers is $\left ( \tfrac{1}{1-z-z^2} \right)^n.$
Using this as $g(z)$, following Proposition 24 below, we can construct a convolved Fibonacci pseudo-involution.

\begin{proposition} If $(g(z), f(z))$ is a pseudo-involution, then $(g^n(z), f(z))$ is a pseudo-involution. \end{proposition}
\begin{proof}
Let $F(z) = -f(z)$. Notice that the Riordan matrix $(g(z), f(z))$ is a pseudo-involution $\iff (g(z), F(z))$ is an involution. This means that $g(z) \cdot g(F(z)) = 1$ and $F(F(z)) = z$. To show that $(g^n(z), F(z))$ is an involution, we need to show that $g^n(z) \cdot g^n(F(z)) = 1$ since we already have that $F(F(z)) = z.$
 Now, \[g^n(z) \cdot g^n(F(z)) = (g(z) \cdot g(F(z)))^n = 1^n = 1.\] 
\end{proof}

The above proposition confirms that    $\left ( (\tfrac{1}{1-z-z^2})^n, \tfrac{1-z-z^2- \sqrt{5z^4+10z^3-z^2-6z+1}}{2-2z-2z^2} \right )$ is a pseudo-involution. Recall the case for $n =1$ is a pseudo-involution given earlier. When $n=2$, we get \[ (g^2(z), f(z)) = \left ( \tfrac{1}{(1-z-z^2)^2}, \tfrac{1-z-z^2- \sqrt{5z^4+10z^3-z^2-6z+1}}{2-2z-2z^2} \right ).\]  The first few entries of the Riordan matrix are shown below, where the convolved Fibonacci numbers \seqnum{A001629} \cite{sloane} are given in the leftmost column:

$$
\begin{bmatrix}
1 \\
2&1 \\
5&5&1 \\
10&20&8&1\\
20&75&44&11&1\\
38&285&212&77&14&1\\
71&1138&976&448&119&17&1\\
130&4820&4476&2390&810&170&20&1\\
235&21545&20838&12266&4905&1325&230&23&1
\end{bmatrix}.
$$
\vspace{0.1in}

The Z-sequence of the convolved Fibonacci pseudo-involution is infinite and starts with \[2, 1, -5, 20, -77, 308, -1303, 5805,...\]
\hspace{0.1in} The A-sequence is also infinite and starts with \[1, 3, 0, 5, -15, 70, -310, 1455, ...\] Notice that this is the same A sequence as that of the Fibonacci pseudo-involution given in \cite{marshall1}.

The modified Lucas and convolved Fibonacci pseudo-involutions are new arrays and none of their columns, other than the first, are sequences found in OEIS. Their row sums, alternating row sums and diagonal sums are also not in OEIS.

\subsection{MATLAB Algorithm for Constructing Pseudo-involutions Starting with g(z)}

The following MATLAB algorithm is used to find the corresponding $f(z)$ that will make $(g(z), f(z))$ a pseudo-involution,  after entering a bi-invertible generating function $g(z)$. Note that a bi-invertible generating function $g(z)$ satisfies $g_0 \neq 0$ and $g_1 \neq 0$. 

\vspace{0.1in}
\noindent syms z \\ \\
g(z) = input('Enter g(z): '); 
$G(z)=g(z)-1 \\ \\
G_(z) = finverse(G)$  \\  \\
R=-G(z)/g(z)  \\  \\
$F(z) = subs(G_(z), z, R)$  \\   \\
$simplify(F(z))  \\   \\
f(z)=-subs(G_(z),z,R) \\  \\
simplify(f(z))$. \\

Construction of pseudo-involutions starting with bi-invertible generating functions $g(z)$ involving the Motzkin \seqnum{A001006} \cite{sloane} and Hex numbers \seqnum{A003215} \cite{sloane} can be found in \cite{marshall1}.

\section{Constructing Pseudo-Involutions Starting with f(z)}

Earlier we constructed pseudo-involutions starting with $g(z)$. We saw that if we were given $g(z)$ such that $g_0 = \pm 1$ and $g_1 \neq 0$, then we can find a unique $f(z)$ such that $(g(z), f(z))$ is a pseudo-involution. 
Now, we will construct pseudo-involutions starting with $f(z)$. Note that this time we are restricted because we need to start with an $f(z)$ such that $-f(-f(z)) = z$.

\begin{theorem}If $f(z)$ is a formal power series such that $-f(z)$ has compositional order $2$, then $\{g(z) \in F_0[[z]] \mid (g(z), f(z))$ is a pseudo-involution$\}$ is an infinite subgroup of $F_0[[z]]$. \end{theorem}

\begin{proof} By Propositions 16,17,18 and 19, $g(z) = f(z)/z,  g(z) = 1$,  $g(z) = f'(z)$  and,  \\ $g(z) = \tfrac{zf'(z)}{f(z)}$ all make $(g(z), f(z))$ a pseudo-involution. These are just a few of the $g(z)$'s that make $(g(z), f(z))$ a pseudo-involution. 

Let $F(z) = -f(z)$. We first show closure under multiplication by showing that if $(g_1(z), F(z))$ and $(g_2(z), F(z))$ are involutions, then $(g_1(z)g_2(z), F(z))$ is also an involution. 
Let $(g_1(z), F(z))$ and $(g_2(z), F(z))$ be involutions. Then, 
\begin{align*}
(g_1(z)g_2(z), F(z))(g_1(z)g_2(z), F(z)) & = (g_1(z)g_2(z) \cdot g_1(F(z)) \cdot g_2(F(z)), F(F(z))) \\
& = ( g_1(z)g_1(F(z)) \cdot g_2(z)g_2(F(z)), F(F(z))) \\
& = (1, z).
\end{align*}
Next, we show closure under taking of inverses by showing that $(g_1^{-1}(z), F(z))$ is also an involution. Thus,
\begin{align*}
(g_1^{-1}(z), F(z))(g_1^{-1}(z), F(z)) & = (g_1^{-1}(z) \cdot g_1^{-1}(F(z)),  F(F(z))) \\
& = \left ( \tfrac{1}{g_1(z)}\tfrac{1}{g_1(F(z))}, F(F(z)) \right ) \\
&= \left ( \tfrac{1}{g_1(z)g_1(F(z))}, F(F(z)) \right ) \\
&= (1, z).
\end{align*}
\end{proof}

\subsection{Construction of Fibonacci-type Pseudo-Involutions in the Riordan Group}

We will now use the $f(z)$ that made the generating function for the Fibonacci numbers and the generating function of the convolved Fibonacci numbers into pseudo-involutions. Recall below, $f(z)$ as given in Subsection 5.2.
\[f(z) = \tfrac{1-z-z^2- \sqrt{5z^4+10z^3-z^2-6z+1}}{2-2z-2z^2}.\]
We know that if we let $g(z) = 1, g(z) = f(z)/z$ or $g(z) = f'(z)$, then we have constructed pseudo-involutions. We start by illustrating the first few rows of $(1, f(z))$. 

$$
\begin{bmatrix}
1\\
0&1\\
0&3&1\\
0&9&6&1\\
0&32&27&9&1\\
0&126&118&54&12&1\\
0&538&525&285&90&15&1\\
0&2429&2408&1440&560&135&18&1\\
0&11412&11378&7203&3195&970&189&21&1\\
0&55201&55146&36162&17488&6195&1542&252&24&1
\end{bmatrix}
$$

\vspace{0.2in}

 Next, we construct a pseudo-involution in the Bell subgroup. Now,
\[f(z)/z = \tfrac{1-z-z^2- \sqrt{5z^4+10z^3-z^2-6z+1}}{2z-2z^2-2z^3}.\]
We illustrate the first few rows of $(f(z)/z, f(z))$.

$$
\begin{bmatrix}
1\\
3&1\\
9&6&1\\
32&27&9&1\\
126&118&54&12&1\\
538&525&285&90&15&1\\
2429&2408&1440&560&135&18&1\\
11412&11378&7203&3195&970&189&21&1\\
55201&55146&36162&17488&6195&1542&252&24&1\\
272993&272904&183132&93926&37043&10926&2303&324&27&1
\end{bmatrix}
$$

\vspace{0.2in}

Now, if 
\[f(z) = \tfrac{1-z-z^2- \sqrt{5z^4+10z^3-z^2-6z+1}}{2-2z-2z^2},\]
then, 
\[f'(z) = \tfrac{-2z-1}{(z^2+z-1)\sqrt{5z^4+10z^3-z^2-6z+1}}. \]

 Next, we illustrate the first few rows of $(f'(z), f(z))$, an element of the Derivative subgroup. 

$$
\begin{bmatrix}
1\\
6&1\\
27&9&1\\
128&54&12&1\\
630&295&90&15&1\\
3228&1575&570&135&18&1\\
17003&8428&3360&980&189&21&1\\
91296&45512&19208&6390&1552&252&24&1\\
496809&248157&108486&39348&11151&2313&324&27&1\\
2729930&1364520&610440&234815&74086&18210&3290&405&30&1
\end{bmatrix}
$$

 Finally, we illustrate the first few rows of $\left(\tfrac{zf'(z)}{f(z)}, f(z) \right )$, an element of the Hitting Time subgroup.

$$
\begin{bmatrix}
1\\
3&1\\
9&6&1\\
42&27&9&1\\
201&128&54&12&1\\
1043&630&295&90&15&1\\
5544&3228&1575&570&135&18&1\\
30012&17003&8428&3360&980&189&21&1\\
164281&91296&45512&19208&6390&1552&252&24&1\\
906693&496809&248157&108486&39348&11151&2313&324&27&1
\end{bmatrix}
$$

\hspace{0.2in} Thus far, we have six Riordan pseudo-involutions with the same $f(z)$. There are many more Riordan pseudo-involutions using this same $f(z)$ since we showed earlier that the $g(z)$'s that make $(g(z), f(z))$ a pseudo-involution, form a group under multiplication. We therefore know that we can multiply $g(z)$'s to get others and we can take powers of these $(g(z))$'s.

\section{Conclusion}
In this paper, we presented examples of Riordan arrays that involve the Fibonacci and modified Lucas numbers. There are also other papers on Riordan arrays that involve the Fibonacci numbers \cite{barry, nkwanta1, shap2}. 

Most of the Riordan matrices, as well as the Lucas stochastic array presented in this paper are new and contain many new sequences of integers. The arrays were observed while studying certain algebraic properties of the Riordan group. Finding combinatorial interpretations of the arrays was not the focus of this paper. Thus, the arrays are open for combinatorial interpretations. Constructing a pseudo-involution starting with a modified Lucas generating function $f(z)$ would also be of interest.

\bigskip
\hrule
\bigskip

\noindent {\it Keywords}: Fibonacci numbers,  Lucas numbers, Riordan group, Riordan matrix, Riordan array, pseudo-involution.

\bigskip
\hrule
\bigskip

\noindent (Concerned with sequences 
\seqnum{A000045},
\seqnum{A000204}, 
\seqnum{A130595},
\seqnum{A007318},
\seqnum{A000032},
\seqnum{A003215},
\seqnum{A001006}
and \seqnum{A001629}.)

\bigskip
\hrule


\begin{thebibliography}{99}

\bibitem{barry} P. Barry, \emph{Riordan Arrays: A Primer}, Logic Press, North Carolina, 2016. 

\bibitem{barry2} P. Barry, Riordan Pseudo-Involutions, Continued Fractions and Somos-4 sequences, J. Integer Sequences 22 (2019) Article 19.6.1.


\bibitem{nkwanta} N. Cameron, A. Nkwanta, On some (pseudo) involutions in
the Riordan group, J. Integer Sequences 8 (2005) Article 05.3.7.

\bibitem{cheon} G. Cheon, H. Kim, S. Jin and L. Shapiro, Riordan group involutions and the $\Delta$-sequence, Discrete Applied Mathematics 157 (2009) 1696-1701.

\bibitem{jean} C. Jean-Louis, A. Nkwanta, Some algebraic structure of the Riordan group, 
Linear Algebra and Its Applications 438 (5) (2013)  2018-2035.

\bibitem{jean2} C. Jean-Louis, A Study on the Algebraic Structure of the Riordan Group, Master's Thesis, Morgan State University, Baltimore MD, 2011.

\bibitem{luzon} A. Luzon, M. A. Moron and L. F. Prieto-Martinez, A formula to construct all involutions in Riordan matrix groups, Linear Algebra and its Applications 533 (2017) 397-417.


\bibitem{marshall1} C. Marshall, Construction of pseudo-involutions in the Riordan Group, Doctoral Dissertation, Morgan State University, Baltimore, MD 2017. 

\bibitem{marshall} C. Marshall, Another method of constructing pseudo-involutions in the Riordan group, Congressus Numeratium 229 (2017), 343-351.

\bibitem{Merlini} D. Merlini, D.G Rogers, R. Sprugnoli and M. C. Verri, On some alternative characterizations of Riordan arrays, Can. J. Math.  49 (1997) 301-320.

\bibitem{nkwanta1} A. Nkwanta, L. W. Shapiro, Pell walks and Riordan matrices, Fibonacci Quart. 43(2005) 170-180.

\bibitem{rob}  F. S. Roberts, \emph{Applied Combinatorics}, Prentice Hall, New Jersey, 2005.

\bibitem{rogers} D. G. Rogers, Pascal triangles, Catalan numbers and renewal arrays, Discrete Math 22 (1978) 301-310.

\bibitem{shap} L. Shapiro, S. Getu, W. Woan, L. Woodson, The Riordan
group, Discrete Appl. Math. 34 (1991) 229 - 239.

\bibitem{shap2} L. Shapiro and S. Getu, The Fibonacci-Catalan-Pascal and Riordan connection, 
Mathematics Newsletter 10 (2000) 25-32.

\bibitem{sloane}
OEIS Foundation Inc.~(2017), The On-Line Encyclopedia of Integer Sequences, \url{http://oeis.org}.
\end{thebibliography}
\end{document}